\documentclass[reqno,a4paper,11pt]{amsart}
\usepackage[ngerman, english]{babel}
\usepackage[all,poly]{xy}
\usepackage[mathcal]{eucal}
\usepackage{amsmath,amssymb, amsthm, amsfonts}
\usepackage[pagebackref,colorlinks]{hyperref}
\usepackage{enumerate}
\usepackage{mathdots}
\usepackage{hhline}
\usepackage{graphics}
\usepackage{comment}
\usepackage{stmaryrd}

\theoremstyle{plain}
\newtheorem {lemma}{Lemma}

\newtheorem {theorem}[lemma]{Theorem}

\theoremstyle{definition}
\newtheorem{definition}[lemma]{Definition}

\newtheorem {example}[lemma]{Example}

\newcommand{\N}{\mathbb{N}}
\newcommand{\Z}{\mathbb{Z}}
\newcommand{\X}{\langle X\rangle}
\newcommand{\A}{\mathcal{A}}
\newcommand{\R}{\mathcal{R}}
\newcommand{\lv}{\operatorname{lv}}
\newcommand{\irr}{\operatorname{irr}}

\title[Bases for Kumjian-Pask algebras]{Bases for Kumjian-Pask algebras over standard $k$-graphs}
\author{Raimund Preusser}

\begin{document}
\maketitle
\section{Notation}
We write $\mathbb{N}$ for the set of all natural numbers, including $0$, and $\mathbb{\Z}$ for the set of all integers. For any $k\in \N\setminus\{0\}$ and $m, n\in \Z^k$, $m\leq n$ means $m_i\leq n_i$ for any $1\leq i\leq k$, $m\lor n$ denotes the pointwise maximum and $m\land n$ the pointwise minimum. Further we set $|m|:=m_1+\dots+m_k$. We denote the usual basis of $\Z^k$ by $\{e_i\}$.\par
In a small category $\mathcal{C}$ with object set $\mathcal{C}^0$, we identify objects $v\in \mathcal{C}^0$ with their identity morphisms, and write $\mathcal{C}$ for the set of morphisms. We write $s$ and $r$ for the domain and codomain maps from $\mathcal{C}$ to $\mathcal{C}^0$.
\section{$k$-graphs and Kumjian-Pask algebras}
\begin{definition}[{\sc $k$-graph}]For a positive integer $k$, we view the additive monoid $\mathbb{N}^k$ as a category with one object. A {\it $k$-graph} is a countable category $\Lambda = (\Lambda^0,\Lambda, r, s)$ together with a functor $d : \Lambda\rightarrow \mathbb{N}^k$, called the {\it degree map}, satisfying the following {\it factorization property}:
if $\lambda\in \Lambda$ and $d(\lambda) = m+n$ for some $m, n\in \mathbb{N}^k$, then there are unique $\mu, \nu \in\Lambda$ such
that $d(\mu) = m$, $d(\nu) = n$, and $\lambda= \mu\circ \nu$.
\end{definition}

For $n\in \mathbb{N}^k$, we write $\Lambda^n := d^{-1}(n)$ (note that it follows from the factorization property that the two definitions of $\Lambda^0$ coincide), and call the elements $\lambda$ of $\Lambda^n$ {\it paths of degree $n$ from $s(\lambda)$ to $r(\lambda)$}. For $v\in \Lambda^0$ we write $v\Lambda^n$ for the set of paths of degree $n$ with range $v$. The $k$-graph $\Lambda$ is called {\it row-finite} if $v\Lambda^n$ is finite for every $v\in \Lambda^0$ and $n\in \mathbb{N}^k$; $\Lambda$ {\it has
no sources} if $v\Lambda^n$ is nonempty for every $v\in \Lambda^0$ and $n\in \mathbb{N}^k$. Further we set $\Lambda^{\neq 0}:= \{\lambda\in\Lambda\mid d(\lambda)\neq 0\}$, and for each $\lambda\in\Lambda^{\neq 0}$ we introduce a {\it ghost path} $\lambda^*$ (it is assumed that the map $\lambda\mapsto \lambda^*$ is injective and further that $(\Lambda^{\neq 0})^*\cap\Lambda=\emptyset$ where $(\Lambda^{\neq 0})^*$ denotes the set of all ghost paths). We set $v^*:=v$ for any $v\in \Lambda^0$.
\begin{definition}[{\sc Kumjian-Pask algebra}]\label{2}
Let $\Lambda$ be a row-finite $k$-graph without sources and let $R$ be a commutative ring with $1$. The (associative, not necessarily unital) $R$-algebra presented by the generating set $\Lambda^0\cup\Lambda^{\neq 0}\cup (\Lambda^{\neq 0})^*$ and the
relations
\begin{enumerate}[(KP1)]
\item $\{v\in\Lambda^0\}$ is a family of mutually orthogonal idempotents,
\item for all $\lambda,\mu\in\Lambda^{\neq 0}$ with $r(\mu) = s(\lambda)$, we have
\[\lambda\mu = \lambda\circ\mu,~\mu^*\lambda^*=(\lambda\circ\mu)^*, ~r(\lambda)\lambda = \lambda = \lambda s(\lambda), ~s(\lambda)\lambda^* = \lambda^* = \lambda^*r(\lambda),\]
\item for all $\lambda,\mu\in\Lambda^{\neq 0}$ with $d(\lambda)=d(\mu)$, we have
\[\lambda^*\mu=\delta_{\lambda,\mu}s(\lambda),\]
\item for all $v\in\Lambda^0$ and all $n\in\mathbb{N}^k\setminus\{0\}$, we have
\[v=\sum\limits_{\lambda\in v\Lambda^n}\lambda\lambda^*\]
\end{enumerate}
is called the {\it Kumjian-Pask algebra} defined by $\Lambda$ and is denoted by $KP_R(\Lambda)$.
\end{definition}
\noindent
{\bf Problem:} Find a basis for $KP_R(\Lambda)$.\\

It follows from the lemma below (which is easy to prove, see \cite[Proof of Lemma 3.3]{pchr}), that the elements $\lambda\mu^*~(\lambda,\mu\in\Lambda)$ span $KP_R(\Lambda)$.
\begin{lemma}\label{3}
Let $\lambda,\mu\in\Lambda$. Then
for each $q\geq d(\lambda)\lor d(\mu)$, we have
\[\lambda^*\mu =\sum\limits_{\substack{\lambda\circ\alpha=\mu\circ \beta,\\d(\lambda\circ\alpha)=q}}\alpha\beta^*.\]
\end{lemma}

\section{Standard $k$-graphs}
In Definition $\ref{stgr}$ below we use the following conventions. Let $l\in \N\setminus\{0\}$ and $s,t\in \N$. We denote by $\{1,\dots,l\}^s$ the cartesian product of $s$ copies of $\{1,\dots,l\}$ (we use the convention $\{1,\dots,l\}^0:=\{()\}$ where $()=\emptyset$ is the empty tuple). Further, for any $p=(p_s,\dots,p_1)\in \{1,\dots,l\}^s$ and $q=(q_t,\dots,q_1)\in \{1,\dots,l\}^t$, we define $p\times q:=
(p_s,\dots,p_1,q_t,\dots,q_1)\in \{1,\dots,l\}^{s+t}$ (hence $p\times q=q$ if $s=0$ and $p\times q=p$ if $t=0$).
\begin{definition}[{\sc Standard $k$-graph}]\label{stgr}
Let $k,l\in\N\setminus\{0\}$. Set \[\Lambda^0:=\Z^k,~~ \Lambda:=\{(v,w,p)\in\Z^k\times \Z^k \times \{1,\dots,l\}^{|v-w|}\mid v\geq w\}\]
and define $r, s : \Lambda\rightarrow \Lambda^0$ by $r(v,w,p) := v$ and $s(v,w,p) := w$. Define composition by $(u,v,p)\circ (v,w,q) = (u, w,p\times q)$ and define $d : \Lambda\rightarrow \N^k$ by $d(v,w,p) := v-w$. Then the $k$-graph $\Lambda= (\Lambda^0,\Lambda, r, s, d)$ is called the {\it standard $k$-graph of level $l$}. 
\end{definition}

If $k,l\in\N\setminus\{0\}$, $\Lambda$ is the standard $k$-graph of level $l$ and $\lambda=(v,w,p)\in \Lambda$, then $p$ is called the {\it level vector of $\lambda$} and is denoted by $\lv(\lambda)$. In the following we will always use the indexing $\lv(\lambda)=(\lv(\lambda)_{|d(\lambda)|},\dots,\lv(\lambda)_1)$ of the components of a level vector $\lv(\lambda)$.

\begin{example}
The skeleton (cf. \cite[p. 3615]{pchr}) of the standard $2$-graph of level $2$ looks as follows:\\\\
\[
\xymatrix{\ddots&&\vdots&&\iddots\\
&(1,-1)\ar@/^0.6pc/[r]\ar@/_0.6pc/[r]   & (1,0)\ar@/^0.6pc/[r]\ar@/_0.6pc/[r] & (1,1)&\\
\dots&(0,-1)\ar@/^0.6pc/[r]\ar@/_0.6pc/[r]\ar@{-->}@/^0.6pc/[u]\ar@{-->}@/_0.6pc/[u]  & (0,0) \ar@/^0.6pc/[r]\ar@/_0.6pc/[r]\ar@{-->}@/^0.6pc/[u]\ar@{-->}@/_0.6pc/[u]  & (0,1)\ar@{-->}@/^0.6pc/[u]\ar@{-->}@/_0.6pc/[u]  &\dots\\
&(-1,-1)\ar@/^0.6pc/[r]\ar@/_0.6pc/[r]\ar@{-->}@/^0.6pc/[u]\ar@{-->}@/_0.6pc/[u] & (-1,0) \ar@/^0.6pc/[r]\ar@/_0.6pc/[r] \ar@{-->}@/^0.6pc/[u]\ar@{-->}@/_0.6pc/[u] &(-1,1) \ar@{-->}@/^0.6pc/[u]\ar@{-->}@/_0.6pc/[u] &&\\
\iddots&&\vdots&&\ddots
}.
\]
\\
\end{example}
\section{Bases for Kumjian-Pask algebras over standard $k$-graphs}
In this section $R$ denotes a commutative ring with $1$ and $\Lambda$ the standard $k$-graph of level $l$ for some $k,l\in\N\setminus\{0\}$. For any $v\in \Lambda^0(=\Z^k)$ and $n\in \N^k\setminus \{0\}$ set 
\[\lambda^{v,n}:=(v,v-n,(1,\dots,1))\in v\Lambda^n.\]
If $\mu\in\Lambda$, we sometimes write $\mu\circ\lambda^{\sim,n}$ instead of $\mu\circ\lambda^{s(\mu),n}$.
Further set
\[\hat\A:=\{(\lambda,\mu)\mid\lambda,\mu\in \Lambda^{\neq 0}, s(\lambda)=s(\mu)\}\]
and
\begin{align*}
\A:=\{(\lambda,\mu)\in\hat\A\mid\text{there is no }n\in\N^k\setminus\{0\}\text{ and }\lambda',\mu'\in\Lambda\text{ such that }\lambda=\lambda'\circ\lambda^{\sim,n}, \mu=\mu'\circ\lambda^{\sim,n}\}.
\end{align*}
The proof of the following lemma is straightforward.
\begin{lemma}\label{6}
Let $(\lambda,\mu)\in \hat \A$. Then $(\lambda,\mu)\in \A$ iff either $\lv(\lambda)_1\neq 1$ or $\lv(\mu)_1\neq 1$ or $d(\lambda)\land d(\mu)=0$.
\end{lemma}

If $(\lambda,\mu),(\lambda',\mu')\in \A$, then we write $(\lambda,\mu)\sim(\lambda',\mu')$ iff there are $m,n\in\N^k\setminus\{0\}$ such that $\lambda\circ\lambda^{\sim,m}=\lambda'\circ\lambda^{\sim,n}$ and $\mu\circ\lambda^{\sim,m}=\mu'\circ\lambda^{\sim,n}$. One checks easily that $\sim$ is an equivalence relation on $\A$. We denote the equivalence class of an element $(\lambda,\mu)\in \A$ by $[(\lambda,\mu)]$. For any equivalence class $[(\lambda,\mu)]$ choose a representative $(\lambda^{[(\lambda,\mu)]},\mu^{[(\lambda,\mu)]})$. We denote by $\R$ the subset of $\A$ consisting of all the chosen representatives. 
\begin{lemma}\label{7}
Let $(\lambda,\mu)\in\A$. Then \[[(\lambda,\mu)]=\{(\lambda',\mu')\in \hat \A\mid r(\lambda')=r(\lambda),r(\mu')=r(\mu), \lv(\lambda')=\lv(\lambda), \lv(\mu')=\lv(\mu)\}.\]
If $d(\lambda)\land d(\mu)=0$, then $[(\lambda,\mu)]=\{(\lambda,\mu)\}$.
\end{lemma}
\begin{proof}$~$\\
$\subseteq$: Let $(\lambda',\mu')\in \A$ such that $(\lambda',\mu')\sim(\lambda,\mu)$. Then $(\lambda',\mu')\in \hat\A$ and further there are $m,n\in\N^k\setminus\{0\}$ such that 
\begin{equation}
\lambda\circ\lambda^{\sim,m}=\lambda'\circ\lambda^{\sim,n} \text{ and }\mu\circ\lambda^{\sim,m}=\mu'\circ\lambda^{\sim,n}.
\end{equation}
It follows from (1) that $r(\lambda')=r(\lambda)$ and $r(\mu')=r(\mu)$. Assume that $|n|>|m|$. Then (1) implies that $\lv(\lambda)_1,\lv(\mu)_1=1$. Hence, by Lemma \ref{6}, $d(\lambda)\land d(\mu)=0$. But, by (1), $d(\lambda)\land d(\mu)=d(\lambda')+n-m\land d(\mu')+n-m$. It follows that $n\leq m$ (if there were an $i\in\{1,\dots,k\}$ such that $n_i>m_i$, then $min(d(\lambda')_i+n_i-m_i,d(\mu')_i+n_i-m_i)\geq n_i-m_i>0~\lightning$). But this contradicts the assumption $|n|>|m|$. Hence $|n|\leq |m|$. By symmetry we also get $|m|\leq |n|$ and hence $|n|=|m|$. Now it follows from (1) that $\lv(\lambda')=\lv(\lambda)$ and $\lv(\mu')=\lv(\mu)$.\\
\\
$\supseteq:$ Let $(\lambda',\mu')\in \hat \A$ such that $r(\lambda')=r(\lambda)$, $r(\mu')=r(\mu)$, $\lv(\lambda')=\lv(\lambda)$ and $\lv(\mu')=\lv(\mu)$. By Lemma \ref{6}, $\lv(\lambda)_1\neq 1$ or $\lv(\mu)_1\neq 1$ or $d(\lambda)\land d(\mu)=0$.\\
\\
\underline{case 1} {\it Assume that $\lv(\lambda)_1\neq 1$ or $\lv(\mu)_1\neq 1$.}\\
Then $(\lambda',\mu')\in \A$ by Lemma \ref{6}. Set $m:=d(\lambda')$ and $n:=d(\lambda)$. One checks easily that $\lambda\circ\lambda^{\sim,m}=\lambda'\circ\lambda^{\sim,n}$ and $\mu\circ\lambda^{\sim,m}=\mu'\circ\lambda^{\sim,n}$. Hence $(\lambda',\mu')\sim(\lambda,\mu)$.\\
\\
\underline{case 2} {\it Assume that $d(\lambda)\land d(\mu)=0$.}\\
Clearly $\lv(\lambda')=\lv(\lambda)$, $\lv(\lambda)\in\{1,\dots,l\}^{|r(\lambda)-s(\lambda)|}=\{1,\dots,l\}^{|r(\lambda)|-|s(\lambda)|}$ and $\lv(\lambda')\in\{1,\dots,l\}^{|r(\lambda')-s(\lambda')|}=\{1,\dots,l\}^{|r(\lambda)|-|s(\lambda')|}$ imply that 
\begin{equation}
|s(\lambda')|=|s(\lambda)|.
\end{equation}
Further 
\begin{equation}
d(\lambda)\land d(\mu)=0 \Leftrightarrow r(\lambda)-s(\lambda)\land r(\mu)-\underbrace{s(\mu)}_{=s(\lambda)}=0\Leftrightarrow r(\lambda)\land r(\mu)=s(\lambda).
\end{equation}
It follows that 
\begin{equation}
s(\lambda')\leq r(\lambda')\land r(\mu')=r(\lambda)\land r(\mu)\overset{(3)}{=}s(\lambda). 
\end{equation}
Hence $s(\mu')=s(\lambda')\overset{(2),(4)}{=}s(\lambda)=s(\mu)$. Thus $(\lambda',\mu')=(\lambda,\mu)$.\\
\\
It remains to show that $[(\lambda,\mu)]=\{(\lambda,\mu)\}$ if $d(\lambda)\land d(\mu)=0$. Let $(\lambda',\mu')\in \A$ such that $(\lambda',\mu')\sim(\lambda,\mu)$. Then, by the inclusion ``$\subseteq$'' shown above, $r(\lambda')=r(\lambda)$, $r(\mu')=r(\mu)$, $\lv(\lambda')=\lv(\lambda)$ and $\lv(\mu')=\lv(\mu)$. It follows from case 2 right above that $(\lambda',\mu')=(\lambda,\mu)$.
\end{proof}

If $(\lambda,\mu)\in \A$ such that $d(\lambda)\land d(\mu)\neq 0$, it can happen that $[(\lambda,\mu)]$ has more than one element. For example let $k,l=2$, $\lambda=((1,1),(1,0),(2))$ and $\lambda'=((1,1),(0,1),(2))$. Then $(\lambda,\lambda),(\lambda',\lambda')\in \A$ by Lemma \ref{6} and $(\lambda,\lambda)\sim(\lambda',\lambda')$ by Lemma \ref{7}. The next lemma shows that the terms $\lambda\mu^*$, where $(\lambda,\mu)$ ranges over a given equivalence class, agree in $KP_R(\Lambda)$.
\begin{lemma}\label{8}
Let $(\lambda,\mu),(\lambda',\mu')\in \A$ such that $(\lambda,\mu)\sim(\lambda',\mu')$. Then $\lambda\mu^*=\lambda'(\mu')^*$ in $KP_R(\Lambda)$.
\end{lemma}
\begin{proof}
Since $(\lambda,\mu)\sim(\lambda',\mu')$, there are $m,n\in\N^k\setminus\{0\}$ such that $\lambda\circ\lambda^{\sim,m}=\lambda'\circ\lambda^{\sim,n}$ and $\mu\circ\lambda^{\sim,m}=\mu'\circ\lambda^{\sim,n}$. Hence
 \begin{align*}
 &\lambda\mu^*-\sum\limits_{\substack{\xi\in s(\lambda)\Lambda^m,\\\xi\neq \lambda^{s(\lambda),m}}}(\lambda\circ\xi)(\mu\circ\xi)^*\\
=&\lambda s(\lambda)\mu^*-\sum\limits_{\substack{\xi\in s(\lambda)\Lambda^m,\\\xi\neq \lambda^{s(\lambda),m}}}\lambda\xi\xi^*\mu^*\\
 =&\lambda(s(\lambda)-\sum\limits_{\substack{\xi\in s(\lambda)\Lambda^m,\\\xi\neq \lambda^{s(\lambda),m}}}\xi\xi^*)\mu^*\\
 =&\lambda\lambda^{\sim,m}(\lambda^{\sim,m})^*\mu^*\\
 =&(\lambda\circ\lambda^{\sim,m})(\mu\circ\lambda^{\sim,m})^*\\
=&(\lambda'\circ\lambda^{\sim,n})(\mu'\circ\lambda^{\sim,n})^*\\
=&\lambda'\lambda^{\sim,n}(\lambda^{\sim,n})^*(\mu')^*\\
 =&\lambda'(s(\lambda')-\sum\limits_{\substack{\zeta\in s(\lambda')\Lambda^n,\\\zeta\neq \lambda^{s(\lambda'),n}}}\zeta\zeta^*)(\mu')^*\\
 =&\lambda' s(\lambda')(\mu')^*-\sum\limits_{\substack{\zeta\in s(\lambda')\Lambda^n,\\\zeta\neq \lambda^{s(\lambda'),n}}}\lambda'\zeta\zeta^*(\mu')^*\\
=&\lambda'(\mu')^*-\sum\limits_{\substack{\zeta\in s(\lambda')\Lambda^n,\\\zeta\neq \lambda^{s(\lambda'),n}}}(\lambda'\circ\zeta)(\mu'\circ\zeta)^*
 \end{align*}
by (KP2) and (KP4). One checks easily that \[\sum\limits_{\substack{\xi\in s(\lambda)\Lambda^m,\\\xi\neq \lambda^{s(\lambda),m}}}(\lambda\circ\xi)(\mu\circ\xi)^*=\sum\limits_{\substack{\zeta\in s(\lambda')\Lambda^n,\\\zeta\neq \lambda^{s(\lambda'),n}}}(\lambda'\circ\zeta)(\mu'\circ\zeta)^*.\]
Thus $\lambda\mu^*=\lambda'(\mu')^*$.
\end{proof}

In the proof of Theorem \ref{14} we will use the definitions and lemmas below.
\begin{definition}\label{9}
Let $s,t\in \N$, $p=(p_s,\dots,p_1)\in \{1,\dots,l\}^s$ and $q=(q_t,\dots,q_1)\in \{1,\dots,l\}^t$. Then we write $p\sim q$ iff $p_{s-i}=q_{t-i}$ for any $i\in\{0,\dots,(s\land t)-1\}$.
\end{definition}
\begin{definition}\label{10}
For any $v,w\in \Lambda^0$, $m,n\in \N^k$, $p\in \{1,\dots,l\}^s$ and $q\in \{1,\dots,l\}^t$ where $s,t\in \N$ such that $|m|-s=|n|-t\geq 0$ define
\begin{align*}
S(v,w,m,n,p,q):=\{(\alpha,\beta)\in v\Lambda^{m}\times w\Lambda^{n}\mid \lv(\alpha)=p \times r,\lv(\beta)=q \times r \text{ for some }r\in \{1,\dots,l\}^{|m|-s}\}.
\end{align*}
\end{definition}
\begin{lemma}\label{11}
Let $\lambda, \mu\in \Lambda^{\neq 0}$ such that $r(\lambda)=r(\mu)$. 
Set
\[S(\lambda,\mu):=\{(\alpha,\beta)\in \Lambda\times \Lambda\mid\lambda\circ\alpha=\mu\circ\beta,~d(\lambda\circ\alpha)=d(\lambda)\lor d(\mu)\}.\]
Then 
\begin{align*}
S(\lambda,\mu)=&S(s(\lambda),s(\mu),(d(\mu)-d(\lambda))\lor 0,(d(\lambda)-d(\mu))\lor 0,\\
&(\lv(\mu)_{|d(\mu)|-|d(\lambda)|},\dots,\lv(\mu)_1),(\lv(\lambda)_{|d(\lambda)|-|d(\mu)|},\dots,\lv(\lambda)_1))
\end{align*}
if $\lv(\lambda)\sim  \lv(\mu)$ and $S(\lambda,\mu)=\emptyset$ otherwise.
\end{lemma}
\begin{proof}
Straightforward.
\end{proof}
\begin{lemma}\label{12}
Let $v,w,m,n,p,q,s,t$ be as in Definition \ref{10}. Further let $\hat n\in \N^k$ be such that $\hat n\leq m\land n$ and $|\hat n|\leq |m|-s$. Then
\[\sum\limits_{(\alpha,\beta)\in S(v,w,m,n,p,q)}\alpha\beta^*=\sum\limits_{(\alpha,\beta)\in S(v,w,m-\hat n,n-\hat n,p,q)}\alpha\beta^*.\]
\end{lemma}
\begin{proof}
Follows from (KP4).
\end{proof}
\begin{lemma}\label{13}
Let $n\in\N^k\setminus\{0\}$, $\lambda,\mu\in \Lambda$ and $v\in \Lambda^0$ such that $s(\lambda)=v=s(\mu)$. Let $i_1,\dots,i_{|n|}\in \{1,\dots,k\}$ such that $n=e_{i_1}+\dots+e_{i_{|n|}}$. Then
\[\sum\limits_{\substack{\xi\in v\Lambda^n,\\\xi\neq \lambda^{v,n}}}(\lambda\circ\xi)(\mu\circ\xi)^*=\sum\limits_{\substack{1\leq p\leq |n|,\\2\leq q\leq l}}(\lambda\circ\xi_{p,q})(\mu\circ\xi_{p,q})^*\]
where $\xi_{p,q}\in v\Lambda^{\sum\limits_{j=0}^{p-1}e_{i_{|n|-j}}}$ and $\lv(\xi_{p,q})=(1,\dots ,1,q)$ for any $1\leq p\leq |n|$ and $2\leq q\leq l$.
\end{lemma}
\begin{proof}
We will prove the lemma by induction on $m:=|n|$.\\
\\
\underline{$m=1$} Obvious.\\
\\
\underline{$m\rightarrow m+1$} Suppose that $|n|=m+1$. Set $n':=e_{i_1}+\dots+e_{i_m}$. Then $n'+e_{i_{m+1}}=n$ and $|n'|=m$. Further set $v':=v-e_{i_{m+1}}$. Then
\begin{align*}
&\sum\limits_{\substack{\xi\in v\Lambda^n,\\\xi\neq \lambda^{v,n}}}(\lambda\circ\xi)(\mu\circ\xi)^*\\
=\quad&\sum\limits_{\substack{\xi_1\in v\Lambda^{e_{i_{m+1}}},~\xi_2\in v'\Lambda^{n'},\\\xi_1\circ\xi_2\neq \lambda^{v,n}}}(\lambda\circ\xi_1\circ\xi_2)(\mu\circ\xi_1\circ\xi_2)^*\\
=\quad&\sum\limits_{\substack{\xi_1\in v\Lambda^{e_{i_{m+1}}},~\xi_2\in v'\Lambda^{n'},\\\xi_1\neq \lambda^{v,e_{i_{m+1}}},~\xi_2\neq \lambda^{v',n'}}}(\lambda\circ\xi_1\circ\xi_2)(\mu\circ\xi_1\circ\xi_2)^*+\sum\limits_{\substack{\xi_1\in v\Lambda^{e_{i_{m+1}}},\\\xi_1\neq \lambda^{v,e_{i_{m+1}}}}}(\lambda\circ\xi_1\circ\lambda^{\sim,n'})(\mu\circ\xi_1\circ\lambda^{\sim,n'})^*\\
&+\sum\limits_{\substack{\xi_2\in v'\Lambda^{n'},\\\xi_2\neq \lambda^{v',n'}}}(\lambda\circ\lambda^{\sim,e_{i_{m+1}}}\circ\xi_2)(\mu\circ\lambda^{\sim,e_{i_{m+1}}}\circ\xi_2)^*
\end{align*}
\begin{align*}
\underset{(KP4)}{\overset{I.A.}{=}}\hspace{0.08cm}&\sum\limits_{\substack{\xi_1\in v\Lambda^{e_{i_{m+1}}},~\xi_2\in v'\Lambda^{n'},\\\xi_1\neq \lambda^{v,e_{i_{m+1}}},~\xi_2\neq \lambda^{v',n'}}}(\lambda\circ\xi_1\circ\xi_2)(\mu\circ\xi_1\circ\xi_2)^*\\
&+\sum\limits_{\substack{\xi_1\in v\Lambda^{e_{i_{m+1}}},\\\xi_1\neq \lambda^{v,e_{i_{m+1}}}}}((\lambda\circ\xi_1)(\mu\circ\xi_1)^*-\sum\limits_{\substack{\xi_2\in v'\Lambda^{n'},\\\xi_2\neq \lambda^{v',n'}}}(\lambda\circ\xi_1\circ\xi_2)(\mu\circ\xi_1\circ\xi_2)^*)\\
&+\sum\limits_{\substack{1\leq p\leq |n'|,\\2\leq q\leq l}}(\lambda\circ\lambda^{\sim,e_{i_{m+1}}}\circ\xi'_{p,q})(\mu\circ\lambda^{\sim,e_{i_{m+1}}}\circ\xi'_{p,q})^*\\
=\quad&\sum\limits_{\substack{\xi_1\in v\Lambda^{e_{i_{m+1}}},\\\xi_1\neq \lambda^{v,e_{i_{m+1}}}}}(\lambda\circ\xi_1)(\mu\circ\xi_1)^*+\sum\limits_{\substack{1\leq p\leq |n'|,\\2\leq q\leq l}}(\lambda\circ\lambda^{\sim,e_{i_{m+1}}}\circ\xi'_{p,q})(\mu\circ\lambda^{\sim,e_{i_{m+1}}}\circ\xi'_{p,q})^*
\end{align*}
where $\xi'_{p,q}\in v'\Lambda^{\sum\limits_{j=0}^{p-1}e_{i_{|n'|-j}}}$ and $\lv(\xi'_{p,q})=(1,\dots ,1,q)$ for any $1\leq p\leq |n'|$ and $2\leq q\leq l$. Clearly $v\Lambda^{e_{i_{m+1}}}\setminus\{\lambda^{v,e_{i_{m+1}}}\}=\{\xi_{1,q}\mid 2\leq q \leq l\}$. Further $\lambda^{v,e_{i_{m+1}}}\circ\xi'_{p,q}=\xi_{p+1,q}$ for any $1\leq p\leq |n'|$ and $2\leq q\leq l$ since $v'-\sum\limits_{j=0}^{p-1}e_{i_{|n'|-j}}=v-\sum\limits_{j=-1}^{p-1}e_{i_{|n'|-j}}=v-\sum\limits_{j=0}^{p}e_{i_{|n|-j}}$. Thus
\begin{align*}
&\sum\limits_{\substack{\xi_1\in v\Lambda^{e_{i_{m+1}}},\\\xi_1\neq \lambda^{v,e_{i_{m+1}}}}}(\lambda\circ\xi_1)(\mu\circ\xi_1)^*+\sum\limits_{\substack{1\leq p\leq |n'|,\\2\leq q\leq l}}(\lambda\circ\lambda^{\sim,e_{i_{m+1}}}\circ\xi'_{p,q})(\mu\circ\lambda^{\sim,e_{i_{m+1}}}\circ\xi'_{p,q})^*\\
=&\sum\limits_{2\leq q \leq l}(\lambda\circ\xi_{1,q})(\mu\circ\xi_{1,q})^*+\sum\limits_{\substack{1\leq p\leq |n'|,\\2\leq q\leq l}}(\lambda\circ\xi_{p+1,q})(\mu\circ\xi_{p+1,q})^*\\
=&\sum\limits_{2\leq q \leq l}(\lambda\circ\xi_{1,q})(\mu\circ\xi_{1,q})^*+\sum\limits_{\substack{2\leq p\leq |n|,\\2\leq q\leq l}}(\lambda\circ\xi_{p,q})(\mu\circ\xi_{p,q})^*\\
=&\sum\limits_{\substack{1\leq p\leq |n|,\\2\leq q\leq l}}(\lambda\circ\xi_{p,q})(\mu\circ\xi_{p,q})^*.
\end{align*}
\end{proof}
\begin{theorem}\label{14}
The elements $v~(v\in \Lambda^0)$, $\lambda~(\lambda\in \Lambda^{\neq 0})$, $\lambda^*~(\lambda\in \Lambda^{\neq 0})$ and $\lambda\mu^*~((\lambda,\mu)\in \R)$ form a basis for $KP_R(\Lambda)$.
\end{theorem}
\begin{proof}
We want to apply \cite[Theorem 15]{hp}. Set $X:=\Lambda^0\cup\Lambda^{\neq 0}\cup (\Lambda^{\neq 0})^*$ and denote by $R\X$ the free $R$-algebra generated by $X$. Consider the relations 
\begin{enumerate}[(1)]
\item for all $\lambda,\mu\in\Lambda$ with $s(\lambda) = r(\mu)$,
\[\lambda\mu = \lambda\circ\mu,~\mu^*\lambda^*=(\lambda\circ\mu)^*,\]
\item for all $\lambda,\mu\in \Lambda$,
\begin{align*}
\lambda\mu=0\text{ if  } s(\lambda)\neq r(\mu),~~\lambda^*\mu=0\text{ if  } r(\lambda)\neq r(\mu),~~\lambda\mu^*=0\text{ if  } s(\lambda)\neq s(\mu),~~\lambda^*\mu^*=0\text{ if  } r(\lambda)\neq s(\mu),
\end{align*}
\item for all $\lambda,\mu\in\Lambda^{\neq 0}$ such that $r(\lambda)=r(\mu)$,
\[\lambda^*\mu =\sum\limits_{(\alpha,\beta)\in S(\lambda,\mu)}\alpha\beta^*,\]
\item for all $v\in\Lambda^0$, $n\in\N^k\setminus\{0\}$, and $\lambda,\mu\in\Lambda$ such that $s(\lambda)=s(\mu)=v$,
\[(\lambda\circ\lambda^{v,n})(\mu\circ \lambda^{v,n})^*=\lambda\mu^*-\sum\limits_{\substack{\xi\in v\Lambda^n,\\\xi\neq \lambda^{v,n}}}(\lambda\circ\xi)(\mu\circ \xi)^*\]
and
\item for all $(\lambda,\mu)\in \A\setminus \R$,
\[\lambda\mu^*=\lambda^{[(\lambda,\mu)]}(\mu^{[(\lambda,\mu)]})^*.\]
\end{enumerate}
It follows from Lemma \ref{3} and Lemma \ref{8} that the relations (1)-(5) above generate the same ideal $I$ of $R\X$ as the relations (KP1)-(KP4) in Definition \ref{2}. Denote by $S$ the reduction system for $R\X$ defined by the relations (1)-(5) (i.e., $S$ is the set of all pairs $\sigma=(W_\sigma,f_\sigma)$ where $W_\sigma$ equals the left hand side of an equation in (1)-(5) and $f_\sigma$ the corresponding right hand side). Denote by $\X$ the semigroup of all nonempty words over $X$ and set $\overline{\X}:=X\cup\{\text{empty word}\}$. For any $A=x_1\dots x_n\in \X$ define
\begin{itemize}
\item $l(A):=n$ (the {\it length} of $A$),
\item $e(A):= \sum\limits_{\substack{1\leq i\leq n,\\x_i\in\Lambda^{\neq 0}}}i$ (the {\it entropy of $A$} ),
\item $f(A):=\sum\limits_{\substack{1\leq i\leq n,\\x_i\in\Lambda^{\neq 0}}}|d(x_i)|$  (the {\it degree value of $A$}),
\item $g(A):=\sum\limits_{\substack{1\leq i\leq n,\\x_i\in\Lambda^{\neq 0}}}\#\{j\in \{1,\dots,|d(x_i)|\}\mid \lv(x_i)_j=1\}$ (the {\it $1$-level value of $A$}) and
\item $h(A):= \#\{i\in\{1,\dots,n-1\}\mid x_ix_{i+1}=\lambda\mu^* \text{ for some }(\lambda,\mu)\in \A\setminus\R\}$ (the {\it $\A\setminus\R$-value of $A$}).
\end{itemize}
Define a partial ordering $\leq$ on $\X$ by 
\begin{align*}
&A\leq B\\
\Leftrightarrow &\big [A=B\big ]~\lor~\big[l(A)<l(B)\big]~\lor ~\big[l(A)=l(B)~\land~ e(A)<e(B)\big]~\lor\\&\big[l(A)=l(B)~\land~ e(A)=e(B)~\land f(A)<f(B)\big]~\lor\\
&\big[l(A)=l(B)~\land~e(A)=e(B)~\land f(A)=f(B)~\land~ g(A)<g(B)\big]~\lor\\
&\big[l(A)=l(B)~\land~e(A)=e(B)~\land f(A)=f(B)~\land~ g(A)=g(B)\\
&~\land~\forall C,D\in \overline\X:h(CAD)< h(CBD)\big].
\end{align*}
Then $\leq$ is a semigroup partial ordering on $\X$ compatible with $S$ and the descending chain condition is satisfied. It remains to show that all ambiguities of $S$ are resolvable. Below we list all ambiguities.
\begin{align*}
(1),(1):&~\lambda\mu\xi,\xi^*\mu^*\lambda^*~~(\lambda,\mu,\xi\in\Lambda,s(\lambda) = r(\mu),s(\mu)=r(\xi))\\
(1),(2):&~\lambda\mu\xi,\lambda\mu\zeta^*\text{ etc.}~~(\lambda,\mu,\xi,\zeta\in\Lambda, s(\lambda)= r(\mu),s(\mu)\neq r(\xi),s(\mu)\neq s(\zeta))\\
(1),(3):&~\lambda^*\mu\xi,\zeta^*\lambda^*\mu~~(\lambda,\mu\in\Lambda^{\neq 0},\xi,\zeta\in \Lambda, r(\lambda) = r(\mu),s(\mu)=r(\xi),r(\zeta)=s(\lambda))\\
(1),(4):&~\xi(\lambda\circ\lambda^{\sim,n})(\mu\circ \lambda^{\sim,n})^*,(\lambda\circ\lambda^{\sim,n})(\mu\circ \lambda^{\sim,n})^*\zeta^*~~(n\in\N^k\setminus\{0\},\lambda,\mu,\xi,\zeta\in\Lambda,\\
&\hspace{7cm} s(\lambda) = s(\mu),s(\xi)=r(\lambda),r(\mu)=s(\zeta))\\
(1),(5):&~\xi\lambda\mu^*,\lambda\mu^*\zeta^*~~((\lambda,\mu)\in\A\setminus \R, \xi,\zeta\in\Lambda, s(\xi)=r(\lambda),r(\mu)=s(\zeta))\\
(2),(2):&~\lambda\mu\xi,\lambda\mu\zeta^*\text{ etc.}~~(\lambda,\mu,\xi\in\Lambda, s(\lambda)\neq r(\mu) , s(\mu)\neq r(\xi) , s(\mu)\neq s(\zeta) ) \\
(2),(3):&~\lambda^*\mu\xi,\lambda^*\mu\zeta^*\text{ etc.}~~(\lambda,\mu,\xi,\zeta\in\Lambda,r(\lambda) = r(\mu),s(\mu)\neq r(\xi),s(\mu)\neq s(\zeta))\\
(2),(4):&~(\lambda\circ\lambda^{\sim,n})(\mu\circ \lambda^{\sim,n})^*\xi,(\lambda\circ\lambda^{\sim,n})(\mu\circ \lambda^{\sim,n})^*\zeta^*\text{ etc.}~~(n\in\N^k\setminus\{0\},\lambda,\mu,\xi,\zeta\in\Lambda,\\
&\hspace{8cm} s(\lambda) = s(\mu),r(\mu)\neq r(\xi),r(\mu)\neq s(\zeta))
\end{align*}
\begin{align*}
(2),(5):&~\lambda\mu^*\xi,\lambda\mu^*\zeta^*\text{ etc.}~~((\lambda,\mu)\in\A\setminus \R, \xi,\zeta\in\Lambda,r(\mu)\neq r(\xi),r(\mu)\neq s(\zeta))\\
(3),(3):&~-\\
(3),(4):&~\xi^*(\lambda\circ\lambda^{\sim,n})(\mu\circ \lambda^{\sim,n})^*,(\lambda\circ\lambda^{\sim,n})(\mu\circ \lambda^{\sim,n})^*\zeta~~(n\in\N^k\setminus\{0\},\lambda,\mu\in\Lambda,\xi,\zeta\in\Lambda^{\neq 0},\\
&\hspace{7cm} s(\lambda) = s(\mu),r(\xi)=r(\lambda),r(\mu)=r(\zeta))\\
(3),(5):&~\xi^*\lambda\mu^*,\lambda\mu^*\zeta~~((\lambda,\mu)\in\A\setminus \R,\xi,\zeta\in\Lambda^{\neq 0}, r(\xi)= r(\lambda),r(\mu)=r(\zeta))\\
(4),(4):&~\lambda\mu^*~~(\lambda=\lambda_1\circ\lambda^{\sim,n_1}=\lambda_2\circ\lambda^{\sim,n_2},\mu=\mu_1\circ\lambda^{\sim,n_1}=\mu_2\circ\lambda^{\sim,n_2},\lambda_1,\lambda_2,\mu_1,\mu_2\in\Lambda,\\
&\hspace{4cm} s(\lambda_1)=s(\mu_1), s(\lambda_2)=s(\mu_2),n_1,n_2\in\N^k\setminus\{0\}, n_1\neq n_2)\\
(4),(5):&~-\\
(5),(5):&~-\\
\end{align*}
We will show how to resolve the (3),(4), the (3),(5) and the (4),(4) ambiguities (the other ambiguities are relatively easy to resolve; for example it follows from Lemma \ref{6} and Lemma \ref{7} that the (1),(5) ambiguities can be resolved).\\
\\
\underline{(3),(4):}\\
Let $n\in\N^k\setminus\{0\}$, $\lambda,\mu\in\Lambda$ and $\xi,\zeta\in\Lambda^{\neq 0}$ such that $s(\lambda) = s(\mu)$, $r(\xi)=r(\lambda)$ and $r(\mu)=r(\zeta)$. We will show how to resolve the ambiguity $(\lambda\circ\lambda^{\sim,n})(\mu\circ \lambda^{\sim,n})^*\zeta$ and leave the ambiguity $\xi^*(\lambda\circ\lambda^{\sim,n})(\mu\circ \lambda^{\sim,n})^*$ to the reader. Suppose that $d(\mu)\neq 0$. Then
\xymatrixcolsep{-3pc}
\xymatrixrowsep{4pc}
\[\xymatrix{
&(\lambda\circ\lambda^{\sim,n})(\mu\circ \lambda^{\sim,n})^*\zeta\ar[rd]^-{(3)}\ar[ld]_-{(4)}&
\\
\lambda\mu^*\zeta-\sum\limits_{\substack{\xi\in s(\lambda)\Lambda^n,\\\xi\neq\lambda^{s(\lambda),n}}}(\lambda\circ\xi)(\mu\circ\xi)^*\zeta\ar[d]_-{(3)}&&(\lambda\circ\lambda^{\sim,n})\sum\limits_{(\alpha, \beta)\in S(\mu\circ \lambda^{\sim,n},\zeta)}\alpha\beta^*\ar[d]^-{(1)}
\\
{\begin{array}{c}
\lambda\sum\limits_{(\alpha',\beta')\in S(\mu,\zeta)}\alpha'(\beta')^*
\\
-\sum\limits_{\substack{\xi\in s(\lambda)\Lambda^n,\\\xi\neq\lambda^{s(\lambda),n}}}(\lambda\circ\xi)\sum\limits_{(\alpha'',\beta'')\in S(\mu\circ\xi,\zeta)}\alpha''(\beta'')^*
\end{array}}\ar[d]_-{(1)}&&{\underbrace{\sum\limits_{(\alpha, \beta)\in S(\mu\circ \lambda^{\sim,n},\zeta)}(\lambda\circ\lambda^{\sim,n}\circ\alpha)\beta^*}_{A:=}}
\\
{\begin{array}{c}
\underbrace{\sum\limits_{(\alpha',\beta')\in S(\mu,\zeta)}(\lambda\circ\alpha')(\beta')^*}_{B:=}
\\
-\underbrace{\sum\limits_{\substack{\xi\in s(\lambda)\Lambda^n,\\\xi\neq\lambda^{s(\lambda),n}}}\sum\limits_{(\alpha'',\beta'')\in S(\mu\circ\xi,\zeta)}(\lambda\circ\xi\circ\alpha'')(\beta'')^*}_{C:=}
\end{array}}&&
}
\]
\\
Suppose that $\lv(\mu)\not\sim\lv(\zeta)$. Then $A=B=C=0$ by Lemma \ref{11}. Hence we can assume that $\lv(\mu)\sim\lv(\zeta)$. Denote by $d(\alpha)$ the degree of an $\alpha$ appearing in $A$ (note that the $\alpha$'s appearing in $A$ all have the same degree by Lemma \ref{11}), by $d(\beta)$ the degree of a $\beta$ appearing in $A$, by $d(\alpha')$ the degree of an $\alpha'$ appearing in $B$ and so on. One checks easily that $d(\lambda^{\sim,n}\circ\alpha)=d(\alpha')+\delta=d(\xi\circ\alpha'')$ and $d(\beta)=d(\beta')+\delta=d(\beta'')$ where $\delta=(d(\mu\circ \lambda^{\sim,n})\lor d(\zeta))-(d(\mu)\lor d(\zeta))\geq 0$. It is an easy exercise to show that $A=B+C$ if $\delta=0$. Hence we can assume that $\delta\neq 0$.\\
\\
\underline{case 1} {\it Assume that $|d(\mu\circ\lambda^{\sim,n})|\leq |d(\zeta)|$.}\\
\\
\underline{case 1.1} {\it Assume that $\lv(\mu\circ\lambda^{\sim,n})\not\sim\lv(\zeta)$.}\\
Then, by Lemma \ref{11}, $A=0$ and $C=\sum\limits_{(\alpha'',\beta'')\in S(\mu\circ\xi,\zeta)}(\lambda\circ\xi\circ\alpha'')(\beta'')^*$ where $\xi\in s(\lambda)\Lambda^n$ has the property that $\lv(\mu\circ\xi)\sim\lv(\zeta)$. Further, also by Lemma \ref{11}, one can write $B=\sum\limits_{(\hat\alpha,\hat\beta)\in S_1}\hat\alpha\hat\beta^*$ and $C=\sum\limits_{(\tilde\alpha,\tilde\beta)\in S_2}\tilde\alpha\tilde\beta^*$ where \[S_1=S(r(\lambda),s(\zeta),d(\lambda)+d(\alpha'), d(\beta'),\lv(\lambda)\times (\lv(\zeta)_{|d(\zeta)|-|d(\mu)|},\dots,\lv(\zeta)_1),())\]
and
\[S_2=S(r(\lambda),s(\zeta),d(\lambda\circ\xi)+d(\alpha''), d(\beta''),\lv(\lambda\circ\xi)\times (\lv(\zeta)_{|d(\zeta)|-|d(\mu\circ\xi)|},\dots,\lv(\zeta)_1),()).\]
One checks easily that $\delta\leq (d(\lambda\circ\xi)+d(\alpha''))\land d(\beta'')$ and $|\delta|\leq |d(\lambda\circ\xi)+d(\alpha'')|-s$ where $s=|d(\lambda\circ\xi)|+|d(\zeta)|-|d(\mu\circ\xi)|$. Hence Lemma \ref{12} shows that $C$ reduces to $C':=\sum\limits_{(\tilde\alpha,\tilde\beta)\in S_3}\tilde\alpha\tilde\beta^*$ where
\[S_3=S(r(\lambda),s(\zeta),d(\lambda\circ\xi)+d(\alpha'')-\delta, d(\beta'')-\delta,\lv(\lambda\circ\xi)\times (\lv(\zeta)_{|d(\zeta)|-|d(\mu\circ\xi)|},\dots,\lv(\zeta)_1),()).\]
Clearly $S_1=S_3$ and hence $C'=B$.\\
\\
\underline{case 1.2} {\it Assume that $\lv(\mu\circ\lambda^{\sim,n})\sim\lv(\zeta)$.}\\
Similar to case 1.1.\\
\\
\underline{case 2} {\it Assume that $|d(\mu)|<|d(\zeta)|<|d(\mu\circ\lambda^{\sim,n})|$.}\\
\\
\underline{case 2.1} {\it Assume that $\lv(\mu\circ\lambda^{\sim,n})\not\sim\lv(\zeta)$.}\\
Then, by Lemma \ref{11}, $A=0$ and $C=\sum\limits_{\substack{\xi\in s(\lambda)\Lambda^n,\\\lv(\mu\circ\xi)\sim\lv(\zeta)}}\sum\limits_{(\alpha'',\beta'')\in S(\mu\circ\xi,\zeta)}(\lambda\circ\xi\circ\alpha'')(\beta'')^*$. Further, also by Lemma \ref{11}, one can write $B=\sum\limits_{(\hat\alpha,\hat\beta)\in S_1}\hat\alpha\hat\beta^*$ and $C=\sum\limits_{(\tilde\alpha,\tilde\beta)\in S_2}\tilde\alpha\tilde\beta^*$ where \[S_1=S(r(\lambda),s(\zeta),d(\lambda)+d(\alpha'), d(\beta'),\lv(\lambda)\times (\lv(\zeta)_{|d(\zeta)|-|d(\mu)|},\dots,\lv(\zeta)_1),())\]
and
\[S_2=S(r(\lambda),s(\zeta),d(\lambda\circ\xi)+d(\alpha''), d(\beta''),\lv(\lambda)\times (\lv(\zeta)_{|d(\zeta)|-|d(\mu)|},\dots,\lv(\zeta)_1),()).\]
One checks easily that $\delta\leq (d(\lambda\circ\xi)+d(\alpha''))\land d(\beta'')$ and $|\delta|\leq |d(\lambda\circ\xi)+d(\alpha'')|-s$ where $s=|d(\lambda)|+|d(\zeta)|-|d(\mu)|$. Hence Lemma \ref{12} shows that $C$ reduces to $C':=\sum\limits_{(\tilde\alpha,\tilde\beta)\in S_3}\tilde\alpha\tilde\beta^*$ where
\[S_3=S(r(\lambda),s(\zeta),d(\lambda\circ\xi)+d(\alpha'')-\delta, d(\beta'')-\delta,\lv(\lambda)\times (\lv(\zeta)_{|d(\zeta)|-|d(\mu)|},\dots,\lv(\zeta)_1),()).\]
Clearly $S_1=S_3$ and hence $C'=B$.\\
\\
\underline{case 2.2} {\it Assume that $\lv(\mu\circ\lambda^{\sim,n})\sim\lv(\zeta)$.}\\
Similar to case 2.1.\\
\\
\underline{case 3} {\it Assume that $|d(\zeta)|\leq |d(\mu)|$.}\\
Since reductions are linear maps, it suffices to show that there is a composition $r$ of reductions such that $r(B)=B=r(A+C)$. Clearly $A+C=\sum\limits_{\xi\in s(\lambda)\Lambda^n}\sum\limits_{(\alpha'',\beta'')\in S(\mu\circ\xi,\zeta)}(\lambda\circ\xi\circ\alpha'')(\beta'')^*$. By Lemma \ref{11}, one can write $B=\sum\limits_{(\hat\alpha,\hat\beta)\in S_1}\hat\alpha\hat\beta^*$ and $A+C=\sum\limits_{(\tilde\alpha,\tilde\beta)\in S_2}\tilde\alpha\tilde\beta^*$ where \[S_1=S(r(\lambda),s(\zeta),d(\lambda)+d(\alpha'), d(\beta'),\lv(\lambda),(\lv(\mu)_{|d(\mu)|-|d(\zeta)|},\dots,\lv(\mu)_1))\]
and
\[S_2=S(r(\lambda),s(\zeta),d(\lambda\circ\xi)+d(\alpha''), d(\beta''),\lv(\lambda), (\lv(\mu)_{|d(\mu)|-|d(\zeta)|},\dots,\lv(\mu)_1)).\]
One checks easily that $\delta\leq (d(\lambda\circ\xi)+d(\alpha''))\land d(\beta'')$ and $|\delta|\leq |d(\lambda\circ\xi)+d(\alpha'')|-s$ where $s=|d(\lambda)|$. Hence Lemma \ref{12} shows that $A+C$ reduces to $D:=\sum\limits_{(\tilde\alpha,\tilde\beta)\in S_3}\tilde\alpha\tilde\beta^*$ where
\[S_3=S(r(\lambda),s(\zeta),d(\lambda\circ\xi)+d(\alpha'')-\delta, d(\beta'')-\delta,\lv(\lambda), (\lv(\mu)_{|d(\mu)|-|d(\zeta)|},\dots,\lv(\mu)_1)).\]
Clearly $S_1=S_3$ and hence $D=B$.\\
\\
The case that $d(\mu)=0$ can be treated analogously.\\
\\
\underline{(3),(5):}\\
Let $(\lambda,\mu)\in\A\setminus \R$ and $\xi,\zeta\in\Lambda^{\neq 0}$ such that $r(\xi)= r(\lambda)$ and $r(\mu)=r(\zeta)$. We will show how to resolve the ambiguity $\lambda\mu^*\zeta$ and leave the ambiguity $\xi^*\lambda\mu^*$ to the reader. Set $\lambda':=\lambda^{[(\lambda,\mu)]}$ and $\mu':=\mu^{[(\lambda,\mu)]}$. Clearly
\xymatrixcolsep{1pc}
\xymatrixrowsep{3pc}
\[\xymatrix{
&\lambda\mu^*\zeta\ar[rd]^-{(3)}\ar[ld]_-{(5)}&
\\
\lambda'(\mu')^*\zeta\ar[d]_-{(3)}&&\lambda\sum\limits_{(\alpha, \beta)\in S(\mu,\zeta)}\alpha\beta^*\ar[d]^-{(1)}
\\
\lambda'\sum\limits_{(\alpha',\beta')\in S(\mu',\zeta)}\alpha'(\beta')^*
\ar[d]_-{(1)}&&{\underbrace{\sum\limits_{(\alpha,\beta)\in S(\mu,\zeta)}(\lambda\circ\alpha)(\beta)^*}_{A:=}}
\\
{\underbrace{\sum\limits_{(\alpha',\beta')\in S(\mu',\zeta)}(\lambda'\circ\alpha')(\beta')^*}_{B:=}}
&&
}
\]
Denote by $d(\alpha)$ the degree of an $\alpha$ appearing in $A$, by $d(\beta)$ the degree of a $\beta$ appearing in $A$, by $d(\alpha')$ the degree of an $\alpha'$ appearing in $B$ and by $d(\beta')$ the degree of a $\beta'$ appearing in $B$. If $\lv(\mu')\overset{L. \ref{7}}{=}\lv(\mu)\not\sim\lv(\zeta)$, then $A=0=B$ by Lemma \ref{11}. Hence we can assume that $\lv(\mu)\sim\lv(\zeta)$. By Lemma \ref{11}, $A=\sum\limits_{(\hat\alpha,\hat\beta)\in S_1}\hat\alpha\hat\beta^*$ and $B=\sum\limits_{(\tilde\alpha,\tilde\beta)\in S_2}\tilde\alpha\tilde\beta^*$ where 
\begin{align*}
S_1=&S(r(\lambda),s(\zeta),d(\lambda)+d(\alpha), d(\beta),\\
&\lv(\lambda)\times (\lv(\zeta)_{|d(\zeta)|-|d(\mu)|},\dots,\lv(\zeta)_1),(\lv(\mu)_{|d(\mu)|-|d(\zeta)|},\dots,\lv(\mu)_1))
\end{align*}
and
\begin{align*}
S_2=&S(r(\lambda'),s(\zeta),d(\lambda')+d(\alpha'), d(\beta'),\\
&\lv(\lambda')\times (\lv(\zeta)_{|d(\zeta)|-|d(\mu')|},\dots,\lv(\zeta)_1),(\lv(\mu')_{|d(\mu')|-|d(\zeta)|},\dots,\lv(\mu')_1)).
\end{align*}
Set $m:=(d(\lambda)+d(\alpha))\land d(\beta)$, $o:=|d(\beta)|-((|d(\mu)|-|d(\zeta)|)\lor 0)$, $m':=(d(\lambda')+d(\alpha'))\land d(\beta')$ and $o':=|d(\beta')|-((|d(\mu')|-|d(\zeta)|)\lor 0)$.\\
\\
\underline{case 1} {\it Assume that $|m|\leq o$.}\\
Clearly 
\begin{align*}
&|m|\leq o\\
\Leftrightarrow~&|(d(\lambda)+d(\alpha))\land d(\beta)|\leq |d(\beta)|-((|d(\mu)|-|d(\zeta)|)\lor 0)\quad\quad\big| -|d(\beta)|\\
\Leftrightarrow~&|(d(\lambda)+d(\alpha)-d(\beta))\land 0|\leq -((|d(\mu)|-|d(\zeta)|)\lor 0)\\
\Leftrightarrow~&|(d(\lambda)+d(\zeta)-d(\mu))\land 0|\leq -((|d(\mu)|-|d(\zeta)|)\lor 0)\\
\overset{L. \ref{7}}{\Leftrightarrow}~&|(d(\lambda')+d(\zeta)-d(\mu'))\land 0|\leq -((|d(\mu')|-|d(\zeta)|)\lor 0)\\
\vdots&\\
\Leftrightarrow~&|m'|\leq o'
\end{align*}
Lemma \ref{12} shows that $A$ reduces to $A':=\sum\limits_{(\gamma,\delta)\in S_3}\gamma\delta^*$ and $B$ reduces to $B':=\sum\limits_{(\gamma',\delta')\in S_4}\gamma'(\delta')^*$ where
\begin{align*}
S_3=&S(r(\lambda),s(\zeta),d(\lambda)+d(\alpha)-m, d(\beta)-m,\\
&\lv(\lambda)\times (\lv(\zeta)_{|d(\zeta)|-|d(\mu)|},\dots,\lv(\zeta)_1),(\lv(\mu)_{|d(\mu)|-|d(\zeta)|},\dots,\lv(\mu)_1))
\end{align*}
and
\begin{align*}
S_4=&S(r(\lambda'),s(\zeta),d(\lambda')+d(\alpha')-m', d(\beta')-m',\\
&\lv(\lambda')\times (\lv(\zeta)_{|d(\zeta)|-|d(\mu')|},\dots,\lv(\zeta)_1),(\lv(\mu')_{|d(\mu')|-|d(\zeta)|},\dots,\lv(\mu')_1)).
\end{align*}
Clearly
\begin{align*}
&d(\lambda)+d(\alpha)-m\\
=&d(\lambda)+d(\alpha)-((d(\lambda)+d(\alpha))\land d(\beta))\\
=&d(\lambda)+d(\alpha)+(-(d(\lambda)+d(\alpha))\lor -d(\beta))\\
=&d(\lambda)+(-d(\lambda)\lor (d(\alpha)-d(\beta)))\\
=&d(\lambda)+(-d(\lambda)\lor (d(\zeta)-d(\mu)))\\
=&d(\lambda)-d(\mu)+((d(\mu)-d(\lambda))\lor d(\zeta))\\
\overset{L. \ref{7}}{=}&d(\lambda')-d(\mu')+((d(\mu')-d(\lambda'))\lor d(\zeta))\\
\vdots&\\
=&d(\lambda')+d(\alpha')-m'
\end{align*}
and analogously $d(\beta)-m=d(\beta')-m'$. Hence $S_3=S_4$ and therefore $A'=B'$.\\
\\
\underline{case 2} {\it Assume that $|m|>o$.}\\
Then $|m'|>o'$ (see the previous case). Suppose that $|d(\zeta)|\geq |d(\mu)|$. Then
\[|m|=|(d(\lambda)+d(\alpha))\land d(\beta)|\leq |d(\beta)|=o~\lightning.\]
Hence $|d(\zeta)|<|d(\mu)|$. Choose $n,n'\in \N^k$ such that $n\leq m$, $n'\leq m'$, $|n|=o$ and $|n'|=o'$. Lemma \ref{12} shows that $A$ reduces to $A':=\sum\limits_{(\gamma,\delta)\in S_3}\gamma\delta^*$ and $B$ reduces to $B':=\sum\limits_{(\gamma',\delta')\in S_4}\gamma'(\delta')^*$ where
\begin{align*}
S_3=&(r(\lambda),s(\zeta),d(\lambda)+d(\alpha)-n, d(\beta)-n,\lv(\lambda),(\lv(\mu)_{|d(\mu)|-|d(\zeta)|},\dots,\lv(\mu)_1))
\end{align*}
and
\begin{align*}
S_4=S(r(\lambda'),s(\zeta),d(\lambda')+d(\alpha')-n', d(\beta')-n',\lv(\lambda'),(\lv(\mu')_{|d(\mu')|-|d(\zeta)|},\dots,\lv(\mu')_1)).
\end{align*}
Clearly $S_3=\{(\gamma,\delta)\}$ and $S_4=\{(\gamma',\delta')\}$ where $\gamma\in r(\lambda)\Lambda^{d(\lambda)+d(\alpha)-n}$, $\lv(\gamma)=\lv(\lambda)$, $\delta\in s(\zeta)\Lambda^{d(\beta)-n}$, $\lv(\delta)=(\lv(\mu)_{|d(\mu)|-|d(\zeta)|},\dots,\lv(\mu)_1)$, $\gamma'\in r(\lambda')\Lambda^{d(\lambda')+d(\alpha')-n'}$, $\lv(\gamma')=\lv(\lambda')$, $\delta'\in s(\zeta)\Lambda^{d(\beta')-n'}$ and $\lv(\delta')=(\lv(\mu')_{|d(\mu')|-|d(\zeta)|},\dots,\lv(\mu')_1)$. Clearly $d(\gamma)\neq 0$ since $d(\lambda)\neq 0$ and $d(\gamma')\neq 0$ since $d(\lambda')\neq 0$. Further $d(\delta),d(\delta')\neq 0$ since $|d(\zeta)|<|d(\mu)|=|d(\mu')|$. Hence $(\gamma,\delta), (\gamma',\delta')\subseteq \hat\A$. By Lemma \ref{7}, $d(\lambda)\land d(\mu)\neq 0$ since $(\lambda,\mu)\in \A\setminus \R$. Hence, by Lemma \ref{7}, $\lv(\lambda)_1\neq 1$ or $\lv(\mu)_1\neq 1$. It follows that $\lv(\gamma)_1\neq 1$ or $\lv(\delta)_1\neq 1$. Analogously $\lv(\gamma')_1\neq 1$ or $\lv(\delta')_1\neq 1$. Hence $(\gamma,\delta), (\gamma',\delta')\subseteq \A$ by Lemma \ref{6}. By Lemma \ref{7}, $(\gamma,\delta)\sim(\gamma',\delta')$. Hence, in view of relation (5), $A'=\gamma\delta^*$ and $B'=\gamma'(\delta')^*$ can be reduced to the same element of $R\X$ (namely $\lambda^{[(\gamma,\delta)]}(\mu^{[(\gamma,\delta)]})^*$).\\
\\
\underline{(4),(4):}\\
Let $n_1,n_2\in\N^k\setminus\{0\}$ and $\lambda_1,\lambda_2,\mu_1,\mu_2\in\Lambda$ such that $n_1\neq n_2$, $v_1:=s(\lambda_1)=s(\mu_1)$, $v_2:=s(\lambda_2)=s(\mu_2)$, $\lambda:=\lambda_1\circ\lambda^{\sim,n_1}=\lambda_2\circ\lambda^{\sim,n_2}$ and $\mu:=\mu_1\circ\lambda^{\sim,n_1}=\mu_2\circ\lambda^{\sim,n_2}$. Clearly
\xymatrixcolsep{1pc}
\xymatrixrowsep{3pc}
\[\xymatrix{
&\lambda\mu^*\ar[rd]^-{(4)}\ar[ld]_-{(4)}&
\\
{\underbrace{\lambda_1\mu_1^*-\sum\limits_{\substack{\xi\in v_1\Lambda^{n_1},\\\xi\neq\lambda^{v_1,n_1}}}(\lambda_1\circ\xi_1)(\mu_1\circ\xi_1)^*}_{A:=}}&&{\underbrace{\lambda_2\mu_2^*-\sum\limits_{\substack{\xi\in v_2\Lambda^{n_2},\\\xi\neq\lambda^{v_2,n_2}}}(\lambda_2\circ\xi_2)(\mu_2\circ\xi_2)^*}_{B:=}}
}
\]
Set $m:=d(\lambda)\land d(\mu)$ and $o:=\max\{i\mid \lv(\lambda)_j=1=\lv(\mu)_j~\forall j\in\{1,\dots,i\}\}$. Choose an $\hat n\in \N^k\setminus\{0\}$ such that $\hat n\leq m$ and $|\hat n|=|m|\land o$. By the factorisation property there are uniquely determined $\hat\lambda, \hat \mu\in \Lambda$ such that $\lambda=\hat\lambda\circ \lambda^{\sim,\hat n}$ and $\mu=\hat\mu\circ \lambda^{\sim,\hat n}$. Set $\hat v:=s(\hat\lambda)=s(\hat\mu)$. Further choose $\hat i_1,\dots,\hat i_{|\hat n|}\in \{1,\dots,k\}$ such that $\hat n=e_{\hat i_1}+\dots+e_{\hat i_{|\hat n|}}$.\\
\\
\underline{case 1} {\it Assume that $d(\hat\lambda),d(\hat\mu)\neq 0$.}\\
Set
\[C:=\lambda^{[(\hat\lambda,\hat\mu)]}(\mu^{[(\hat\lambda,\hat\mu)]})^*-\sum\limits_{\substack{1\leq p\leq |\hat n|,\\2\leq q\leq l}}\lambda^{[(\hat\lambda\circ\hat\xi_{p,q},\hat\mu\circ\hat\xi_{p,q})]}(\mu^{[(\hat\lambda\circ\hat\xi_{p,q},\hat\mu\circ\hat\xi_{p,q})]})^*\]
where $\hat\xi_{p,q}\in \hat v\Lambda^{\sum\limits_{j=0}^{p-1}e_{\hat i_{|\hat n|-j}}}$ and $\lv(\hat\xi_{p,q})=(1,\dots ,1,q)$ for any $1\leq p\leq |\hat n|$ and $2\leq q\leq l$. Note that $(\hat\lambda,\hat\mu), (\hat\lambda\circ\xi_{p,q},\hat\mu\circ\xi_{p,q})\in \A$ for any $1\leq p\leq |\hat n|$ and $2\leq q\leq l$ by Lemma \ref{6}. We will show that $A$ can be reduced to $C$. It will follow by symmetry that $B$ also can be reduced to $C$. Clearly $|\hat n|\geq |n_1|$.\\
\\
\underline{case 1.1} {\it Assume that $|\hat n|=|n_1|$.}\\
Choose $i_1,\dots,i_{|n_1|}\in \{1,\dots,k\}$ such that $n_1=e_{i_1}+\dots+e_{i_{|n_1|}}$. Lemma \ref{13} shows that $A$ reduces to 
\[A':=\lambda_1\mu_1^*-\sum\limits_{\substack{1\leq p\leq |n_1|,\\2\leq q\leq l}}(\lambda_1\circ\xi_{p,q})(\mu_1\circ\xi_{p,q})^*\]
where $\xi_{p,q}\in v_1\Lambda^{\sum\limits_{j=0}^{p-1}e_{i_{|n_1|-j}}}$ and $\lv(\xi_{p,q})=(1,\dots ,1,q)$ for any $1\leq p\leq |n_1|$ and $2\leq q\leq l$. 
Clearly $\lambda_1\circ\lambda^{\sim,n_1}=\lambda=\hat\lambda\circ\lambda^{\sim,\hat n}$, $\mu_1\circ\lambda^{\sim,n_1}=\mu=\hat\mu\circ\lambda^{\sim,\hat n}$ and $|\hat n|=|n_1|$ imply that $\lv(\lambda_1)=\lv(\hat\lambda)$ and $\lv(\mu_1)=\lv(\hat\mu)$. It follows that $(\lambda_1,\mu_1)\in \A$ and $(\lambda_1,\mu_1)\sim(\hat\lambda,\hat\mu)$ (by Lemma \ref{7}). Further it follows that for any $1\leq p\leq |n_1|=|\hat n|$ and $2\leq q\leq l$, $(\lambda_1\circ\xi_{p,q},\mu_1\circ\xi_{p,q})\in \A$ and $(\lambda_1\circ\xi_{p,q},\mu_1\circ\xi_{p,q})\sim(\hat\lambda\circ\hat\xi_{p,q},\hat\mu\circ\hat\xi_{p,q})$ (also by Lemma \ref{7}). Hence, in view of relation (5), $A'$ reduces to $C$.\\
\\
\underline{case 1.2} {\it Assume that $|\hat n|>|n_1|$.}\\
\\
\underline{case 1.2.1} {\it Assume that $o\geq|m|$.}\\
Then clearly $\hat n=m(=d(\lambda)\land d(\mu))$ since $\hat n\leq m$ and $|\hat n|=|m|$. Hence $n_1\leq m=\hat n$. Set $\tilde n:=\hat n-n_1\geq 0$. Then $\tilde n\neq 0$ since $|\hat n|>|n_1|$. Clearly $\lambda_1=\hat\lambda\circ\lambda^{\sim,\tilde n}$ and $\mu_1=\hat\mu\circ\lambda^{\sim,\tilde n}$. Hence \[A=\lambda_1\mu_1^*-\sum\limits_{\substack{\xi\in v_1\Lambda^{n_1},\\\xi\neq\lambda^{v_1,n_1}}}(\lambda_1\circ\xi_1)(\mu_1\circ\xi_1)^*=(\hat\lambda\circ\lambda^{\sim,\tilde n})(\hat\mu\circ\lambda^{\sim,\tilde n})^*-\sum\limits_{\substack{\xi\in v_1\Lambda^{n_1},\\\xi\neq\lambda^{v_1,n_1}}}(\lambda_1\circ\xi_1)(\mu_1\circ\xi_1)^*.\]
In view of relation (4), $A$ reduces to
\[A':=\hat\lambda\hat\mu^*-\sum\limits_{\substack{\xi'\in \hat v\Lambda^{\tilde n},\\\xi'\neq\lambda^{\hat v,\tilde n}}}(\hat\lambda\circ\xi')(\hat\mu\circ\xi')^*-\sum\limits_{\substack{\xi\in v_1\Lambda^{n_1},\\\xi\neq\lambda^{v_1,n_1}}}(\lambda_1\circ\xi_1)(\mu_1\circ\xi_1)^*.\]
Choose $i_1,\dots,i_{|n_1|}\in \{1,\dots,k\}$ and $\tilde i_1,\dots,\tilde i_{|\tilde n|}\in \{1,\dots,k\}$ such that $n_1=e_{i_1}+\dots+e_{i_{|n_1|}}$ and $\tilde n=e_{\tilde i_1}+\dots+e_{\tilde i_{|\tilde n|}}$. Lemma \ref{13} shows that $A'$ reduces to 
\[A'':=\hat\lambda\hat\mu^*-\sum\limits_{\substack{1\leq p\leq |\tilde n|,\\2\leq q\leq l}}(\hat\lambda\circ\tilde\xi_{p,q})(\hat\mu\circ\tilde\xi_{p,q})^*-\sum\limits_{\substack{1\leq p\leq |n_1|,\\2\leq q\leq l}}(\lambda_1\circ\xi_{p,q})(\mu_1\circ\xi_{p,q})^*\]
where $\tilde\xi_{p,q}\in\hat v\Lambda^{\sum\limits_{j=0}^{p-1}e_{\tilde i_{|\tilde n|-j}}}$ and $\lv(\tilde\xi_{p,q})=(1,\dots ,1,q)$ for any $1\leq p\leq |\tilde n|$ and $2\leq q\leq l$ and $\xi_{p,q}\in v_1\Lambda^{\sum\limits_{j=0}^{p-1}e_{i_{|n_1|-j}}}$ and $\lv(\xi_{p,q})=(1,\dots ,1,q)$ for any $1\leq p\leq |n_1|$ and $2\leq q\leq l$. It follows from Lemma \ref{7} that for any $1\leq p\leq |\tilde n|$ and $2\leq q\leq l$, $(\hat\lambda\circ\tilde\xi_{p,q},\hat\mu\circ\tilde \xi_{p,q})\in \A$ and $(\hat\lambda\circ\tilde\xi_{p,q},\hat\mu\circ\tilde \xi_{p,q})\sim(\hat\lambda\circ\hat\xi_{p,q},\hat\mu\circ\hat\xi_{p,q})$. Further, also by Lemma \ref{7}, for any $1\leq p\leq |n_1|$ and $2\leq q\leq l$, $(\lambda_1\circ\xi_{p,q},\mu_1\circ\xi_{p,q})\in \A$ and $(\lambda_1\circ\xi_{p,q},\mu_1\circ\xi_{p,q})\sim(\hat\lambda\circ\hat\xi_{|\tilde n|+p,q},\hat\mu\circ\hat\xi_{|\tilde n|+p,q})$. Hence, in view of relation (5), $A''$ reduces to $C$.\\
\\
\underline{case 1.2.2} {\it Assume that $o<|m|$.}\\
Then $|\hat n|=o$. Hence 
\begin{align*}
&|\hat n|<|m|=|d(\lambda)\land d(\mu)|\\
\Rightarrow~&|\hat n|-|n_1|<|d(\lambda)\land d(\mu)|-|n_1|\\
\Rightarrow~&|\hat n|-|n_1|<|d(\lambda_1)\land d(\mu_1)|.
\end{align*}
Therefore one can choose an $\tilde n\in \N^k\setminus\{0\}$ such that $\tilde n\leq d(\lambda_1)\land d(\mu_1)$ and $|\tilde n|=|\hat n|-|n_1|>0$. Clearly $\lambda_1=\tilde\lambda\circ\lambda^{\sim,\tilde n}$ and $\mu_1=\tilde\mu\circ\lambda^{\sim,\tilde n}$ where $\tilde \lambda\in r(\lambda)\Lambda^{d(\lambda)-n_1-\tilde n}$, $\tilde \mu\in r(\mu)\Lambda^{d(\mu)-n_1-\tilde n}$, $\lv(\tilde \lambda)\sim\lv(\lambda)$ and $\lv(\tilde \mu)\sim\lv(\mu)$. Hence \[A=\lambda_1\mu_1^*-\sum\limits_{\substack{\xi\in v_1\Lambda^{n_1},\\\xi\neq\lambda^{v_1,n_1}}}(\lambda_1\circ\xi_1)(\mu_1\circ\xi_1)^*=(\tilde\lambda\circ\lambda^{\sim,\tilde n})(\tilde\mu\circ\lambda^{\sim,\tilde n})^*-\sum\limits_{\substack{\xi\in v_1\Lambda^{n_1},\\\xi\neq\lambda^{v_1,n_1}}}(\lambda_1\circ\xi_1)(\mu_1\circ\xi_1)^*.\]
Set $\tilde v:=s(\tilde\lambda)=s(\tilde\mu)$. In view of relation (4), $A$ reduces to
\[A':=\tilde\lambda\tilde\mu^*-\sum\limits_{\substack{\xi'\in \tilde v\Lambda^{\tilde n},\\\xi'\neq\lambda^{\tilde v,\tilde n}}}(\tilde\lambda\circ\xi')(\tilde\mu\circ\xi')^*-\sum\limits_{\substack{\xi\in v_1\Lambda^{n_1},\\\xi\neq\lambda^{v_1,n_1}}}(\lambda_1\circ\xi_1)(\mu_1\circ\xi_1)^*.\]
Choose $i_1,\dots,i_{|n_1|}\in \{1,\dots,k\}$ and $\tilde i_1,\dots,\tilde i_{|\tilde n|}\in \{1,\dots,k\}$ such that $n_1=e_{i_1}+\dots+e_{i_{|n_1|}}$ and $\tilde n=e_{\tilde i_1}+\dots+e_{\tilde i_{|\tilde n|}}$. Lemma \ref{13} shows that $A'$ reduces to 
\[A'':=\tilde\lambda\tilde\mu^*-\sum\limits_{\substack{1\leq p\leq |\tilde n|,\\2\leq q\leq l}}(\tilde\lambda\circ\tilde\xi_{p,q})(\tilde\mu\circ\tilde\xi_{p,q})^*-\sum\limits_{\substack{1\leq p\leq |n_1|,\\2\leq q\leq l}}(\lambda_1\circ\xi_{p,q})(\mu_1\circ\xi_{p,q})^*\]
where $\tilde\xi_{p,q}\in\tilde v\Lambda^{\sum\limits_{j=0}^{p-1}e_{\tilde i_{|\tilde n|-j}}}$ and $\lv(\tilde\xi_{p,q})=(1,\dots ,1,q)$ for any $1\leq p\leq |\tilde n|$ and $2\leq q\leq l$ and $\xi_{p,q}\in v_1\Lambda^{\sum\limits_{j=0}^{p-1}e_{i_{|n_1|-j}}}$ and $\lv(\xi_{p,q})=(1,\dots ,1,q)$ for any $1\leq p\leq |n_1|$ and $2\leq q\leq l$. It follows from Lemma \ref{7} that $(\tilde\lambda,\tilde\mu)\in \A$ and $(\tilde\lambda,\tilde\mu)\sim(\hat\lambda,\hat\mu)$. Further it follows from Lemma \ref{7} that for any $1\leq p\leq |\tilde n|$ and $2\leq q\leq l$, $(\tilde\lambda\circ\tilde\xi_{p,q},\tilde\mu\circ\tilde \xi_{p,q})\in \A$ and $(\tilde\lambda\circ\tilde\xi_{p,q},\tilde\mu\circ\tilde \xi_{p,q})\sim(\hat\lambda\circ\hat\xi_{p,q},\hat\mu\circ\hat\xi_{p,q})$. Further, also by Lemma \ref{7}, for any $1\leq p\leq |n_1|$ and $2\leq q\leq l$, $(\lambda_1\circ\xi_{p,q},\mu_1\circ\xi_{p,q})\in \A$ and $(\lambda_1\circ\xi_{p,q},\mu_1\circ\xi_{p,q})\sim(\hat\lambda\circ\hat\xi_{|\tilde n|+p,q},\hat\mu\circ\hat\xi_{|\tilde n|+p,q})$. Hence, in view of relation (5), $A''$ reduces to $C$.\\
\\
\underline{case 2} {\it Assume that $d(\hat\lambda)=0$ or $d(\hat\mu)=0$.}\\
This case is very similar to case 1 and hence is omitted.\\
\\
Thus all ambiguities are resolvable. It follows from \cite[Theorem 15]{hp} that $KP_R(\Lambda)=R\X/I$ is isomorphic to $R\X_{\irr}$ as an $R$-module where $R\X_{\irr}$ is the submodule of $R\X$ consisting of all irreducible elements (cf. \cite[Definition 11]{hp}). Clearly the elements $v~(v\in \Lambda^0)$, $\lambda~(\lambda\in \Lambda^{\neq 0})$, $\lambda^*~(\lambda\in \Lambda^{\neq 0})$ and $\lambda\mu^*~((\lambda,\mu)\in \R)$ form a basis for $R\X_{\irr}$.
\end{proof}

\end{document}